\documentclass[10pt, reqno]{amsart}
\usepackage{graphicx, amssymb, amsmath, amsthm, color, slashed}
\numberwithin{equation}{section}

\usepackage{hyperref}

\let\Re=\undefined\DeclareMathOperator*{\Re}{Re}
\let\Im=\undefined\DeclareMathOperator*{\Im}{Im}

\newcommand{\R}{\mathbb{R}}
\newcommand{\C}{\mathbb{C}}

\newcommand{\eps}{\varepsilon}

\newtheorem{theorem}{Theorem}[section]

\newtheorem{lemma}[theorem]{Lemma}

\newtheorem{proposition}[theorem]{Proposition}

\theoremstyle{definition}

\newtheorem{remark}[theorem]{Remark}

\theoremstyle{remark}

\newcommand{\qtq}[1]{\quad\text{#1}\quad}

\def\({\left(}
\def\){\right)}
 
\begin{document}
\title[Focusing radial NLS]{A new proof of scattering below the ground state for the 3d radial focusing cubic NLS}
\author[B. Dodson]{Benjamin Dodson}
\address{Department of Mathematics, Johns Hopkins University}
\email{bdodson4@jhu.edu}
\author[J. Murphy]{Jason Murphy}
\address{Departments of Mathematics, UC Berkeley}
\email{murphy@math.berkeley.edu}

\begin{abstract}
We revisit the scattering result of Holmer and Roudenko \cite{HR} on the radial focusing cubic NLS in three space dimensions.  Using the radial Sobolev embedding and a virial/Morawetz estimate, we give a simple proof of scattering below the ground state that avoids the use of concentration compactness. 
\end{abstract}
\maketitle

\section{Introduction}
We consider the initial-value problem for the focusing cubic nonlinear Schr\"odinger equation (NLS) in three space dimensions:
\begin{equation}\label{nls}
\begin{cases}
(i\partial_t+\Delta) u = -|u|^2 u, \\
u(0)=u_0\in H_x^1(\R^3),
\end{cases}
\end{equation} 
where $u:\R_t\times\R_x^3\to\C$.  Solutions to \eqref{nls} conserve the \emph{mass}, defined by
\[
M(u(t)) := \int_{\R^3} |u(t,x)|^2\,dx
\]
and the \emph{energy}, defined as the sum of the kinetic and potential energies:
\[
E(u(t)) := \int_{\R^3} \tfrac12 |\nabla u(t,x)|^2 -\tfrac14|u(t,x)|^4 \,dx.
\]
We call \eqref{nls} $\dot H^{1/2}$-critical, as the $\dot H^{1/2}$-norm of initial data is invariant under the scaling that preserves the class of solutions, namely,
\[
u(t,x)\mapsto \lambda u(\lambda^2 t,\lambda x).
\]

By a solution to \eqref{nls}, we mean a function $u\in C_t H_x^1(I\times\R^3)$ on some interval $I\ni 0$ that satisfies the Duhamel formula
\[
u(t)=e^{it\Delta}u_0 +i\int_0^t e^{i(t-s)\Delta}(|u|^2 u)(s)\,ds
\]
for $t\in I$, where $e^{it\Delta}$ is the Schr\"odinger group.  If $I=\R$, we call $u$ global.  A global solution $u$ to \eqref{nls} \emph{scatters} if there exist $u_\pm\in H_x^1(\R^3)$ such that
\[
\lim_{t\to\pm\infty} \|u(t)-e^{it\Delta}u_\pm\|_{H_x^1(\R^3)}=0.
\]

Equation \eqref{nls} admits a global but {non-scattering} solution
\[
u(t,x)=e^{it}Q(x),
\] 
where $Q$ is the \emph{ground state}, i.e. the unique positive, decaying solution to the elliptic equation
\begin{equation}\label{elliptic}
-\Delta Q + Q - Q^3 = 0. 
\end{equation}

Holmer and Roudenko \cite{HR} proved the following scattering result:
\begin{theorem}\label{T} Suppose $u_0$ is radial and $M(u_0)E(u_0)<M(Q)E(Q)$. 

If $\|u_0\|_{L_x^2}\|u_0\|_{\dot H_x^1}<\|Q\|_{L_x^2} \|Q\|_{\dot H_x^1}$, then the solution to \eqref{nls} is global and scatters.
\end{theorem}

The proof in \cite{HR} was based on the concentration compactness approach to induction on energy.  In this note, we present a simplified proof of Theorem~\ref{T} that avoids concentration compactness.  We use the radial Sobolev embedding to establish a virial/Morawetz estimate, which in turn implies `energy evacuation' as $t\to\infty$.  Together with a scattering criterion introduced by Tao \cite{Tao} (see Lemma~\ref{criterion}), this suffices to prove Theorem~\ref{T}.

\begin{remark}
 In \cite{DHR}, Theorem~\ref{T} is extended to the non-radial setting, also through the concentration compactness approach.  It remains to be seen if the non-radial case can be treated without concentration compactness.
\end{remark}

\begin{remark} The authors of \cite{HR} proved more than just scattering.  Their proof gives global space-time bounds of the form
\begin{equation}\label{stb}
\|u\|_{L_{t,x}^5(\R\times\R^3)} \leq C\bigl( M(Q)E(Q)-M(u_0)E(u_0) \bigr)
\end{equation}
for some function $C:(0,M(Q)E(Q))\to(0,\infty).$  We use a different scattering criterion introduced by Tao \cite{Tao} (Lemma~\ref{criterion}).  One also obtains space-time bounds; however, as the proof will show, the bounds depend on the profile of the initial data and not just the size.  In this sense, we prove something weaker than the result of \cite{HR}; however, our proof is much simpler.
\end{remark}

\begin{remark} The results of \cite{DHR, HR} were generalized to other dimensions and intercritical nonlinearities by \cite{CFX, Guevara}.  It should be possible to generalize our arguments as well (at least in $d\geq 3$, where the Morawetz estimates should work out).
\end{remark}

\subsection*{Acknowledgements} B.D. was supported by NSF DMS-1500424. J.M. was supported by the NSF Postdoctoral Fellowship DMS-1400706.   This work was carried out while the authors were visiting the IMA at the University of Minnesota, which provided an ideal working environment.

\section{Preliminaries}\label{S:notation} We write $A\lesssim B$ when $A$ and $B$ are non-negative quantities such that $A\leq CB$ for some $C>0$.  We indicate dependence on parameters via subscripts, e.g. $A\lesssim_u B$ indicates $A\leq CB$ for some $C=C(u)>0$.  We also use the big-oh notation $\mathcal{O}$.  We write
\[
\|f\|_{L_x^r} = \biggl(\int_{\R^3} |f(x)|^r\,dx\biggr)^{1/r},\quad \|f\|_{L_t^q L_x^r(I\times\R^3)} = \bigl\| \|f(t)\|_{L_x^r} \bigr\|_{L_t^q(I)},
\]
with the usual adjustments when $q$ or $r$ is $\infty$. We write
\[
\|f\|_{\dot H_x^1} = \|\nabla f\|_{L_x^2},\quad \|f\|_{H_x^1} = \|f\|_{L_x^2} + \|\nabla f\|_{L_x^2},\quad \|f\|_{H_x^{1,r}} = \|f\|_{L_x^r} + \|\nabla f \|_{L_x^r}.  
\]

In this note, we restrict to radial (i.e. spherically-symmetric) solutions.  The following {radial Sobolev embedding} (which may be deduced by the fundamental theorem of calculus and Cauchy--Schwarz) plays a crucial role:
\begin{lemma}[Radial Sobolev embedding] For radial $f\in H^1(\R^3)$,
\[
\| |x| f\|_{L_x^\infty} \lesssim \|f\|_{H_x^1}. 
\]
\end{lemma}

\subsection{Local theory and a scattering criterion}
The local theory for \eqref{nls} is standard.  For any $u_0\in H^1$, there exists a unique maximal-lifespan solution $u:I\times\R^3\to\C$ to \eqref{nls}.  This solution belongs to $C_t H_x^1(I\times\R^3)$ and conserves the mass and energy.  Because \eqref{nls} is $H^1$-subcritical, we have an \emph{$H^1$-blowup criterion}, namely, we can extend the solution as long as its $H^1$-norm stays bounded.  In particular, if $u$ remains uniformly bounded in $H^1$ throughout its lifespan, then it is global.  For a textbook treatment, we refer the reader to \cite{Caz}.   

In \cite{Tao}, Tao established a scattering criterion for radial solutions to \eqref{nls}.  The following appears in \cite[Theorem~1.1]{Tao}; for completeness, we include a sketch of the proof.  We assume familiarity with Strichartz estimates.

\begin{lemma}[Scattering criterion, \cite{Tao}]\label{criterion} Suppose $u:\R_t\times\R_x^3\to\C$ is a radial solution to \eqref{nls} satisfying 
\[
\|u\|_{L_t^\infty H_x^1(\R\times\R^3)} \leq E. 
\]
There exist $\eps=\eps(E)>0$ and $R=R(E)>0$ such that if
\[
\liminf_{t\to\infty} \int_{|x|\leq R} |u(t,x)|^2 \,dx \leq \eps^{2},
\]
then $u$ scatters forward in time. 
\end{lemma} 

\begin{proof}[Sketch of proof] Let $0<\eps<1$ and $R\geq 1$ to be chosen later.  The implicit constants below may depend on $E$.  By Sobolev embedding, Strichartz, and monotone convergence, we may choose $T_0$ large enough depending on $u_0$ so that
\[
\| e^{it\Delta}u_0\|_{L_t^4 L_x^6([T_0,\infty)\times\R^3)} <\eps. 
\]
We may further assume $T_0>\eps^{-1}$.  

By assumption, we may choose $T>T_0$ so that 
\[
\int \chi_R(x) |u(T,x)|^2\,dx\leq \eps^{2},
\]
where $\chi_R$ is a smooth cutoff to $\{|x|\leq R\}$.  Using the identity
\[
\partial_t |u|^2 =-2\nabla\cdot\Im(\bar u \nabla u),
\]
which follows from \eqref{nls}, together with integration by parts and Cauchy--Schwarz, we deduce
\[
\biggl|\partial_t \int \chi_R(x) |u(t,x)|^2\,dx\biggr| \lesssim \tfrac{1}{R}. 
\]
Thus, choosing $R\gg \eps^{-{9}/{4}}$, we find
\[
\|\chi_R u\|_{L_t^\infty L_x^2(I_1\times\R^3)} \lesssim \eps,\qtq{where}I_1 = [T-\eps^{-\frac14},T].
\]
Using H\"older's inequality, Sobolev embedding, and radial Sobolev embedding, we deduce
\begin{align*}
\|u\|_{L_t^\infty L_x^3} & \lesssim \|\chi_R u\|_{L_t^\infty L_x^2}^{\frac12} \|u\|_{L_t^\infty L_x^6}^{\frac12} + \|(1-\chi_R) u\|_{L_{t,x}^\infty}^{\frac13} \|u\|_{L_t^\infty L_x^2}^{\frac23} \lesssim \eps^{\frac12} + R^{-\frac13} \lesssim \eps^{\frac12},
\end{align*}
where all space-time norms are over $I_1\times\R^3$. 

We next use the Duhamel formula to write
\begin{align*}
e^{i(t-T)\Delta}u(T) = e^{it\Delta}u_0 + F_1(t) + F_2(t), \end{align*}
where
\[
F_j(t):= i\int_{I_j} e^{i(t-s)\Delta}(|u|^2 u)(s)\,ds\qtq{and} I_2=[0,T-\eps^{-\frac14}].
\]

It is not difficult to prove by a continuity argument, Sobolev embedding, and Strichartz estimates, that
\[
\| u \|_{L_t^2 L_x^\infty(I\times\R^3)} + \|u\|_{L_t^2 H_x^{1,6}(I\times\R^3)} \lesssim (1+|I|)^{1/2} 
\]
for any interval $I$.  Thus, by Sobolev embedding and Strichartz, 
\[
\|F_1\|_{L_t^4 L_x^6([T,\infty)\times\R^3)}  \lesssim \int_{I_1} \| (|u|^2 u)(s)\|_{H_x^1} \,ds  \lesssim \| u\|_{L_t^\infty L_x^3} \|u\|_{L_t^2 L_x^\infty} \|u\|_{L_t^2 H_x^{1,6}} \lesssim \eps^{\frac14},
\]
where the final three space-time norms are over $I_1\times\R^3$.  

Next, by H\"older's inequality,
\[
\| F_2\|_{L_t^4 L_x^6([T,\infty)\times\R^3)} \lesssim \|F_2\|_{L_t^4 L_x^3([T,\infty)\times\R^3)}^{\frac12}\|F_2\|_{L_t^4 L_x^\infty([T,\infty)\times\R^3)}^{\frac12}.
\]
Noting that
\[
F_2(t) = e^{i(t-T+\eps^{-1/4})\Delta}[u(T-\eps^{-1/4})-u_0], 
\]
we can first use Strichartz to estimate
\[
\|F_2\|_{L_t^4 L_x^3([T,\infty)\times\R^3)} \lesssim 1. 
\]
On the other hand, by the $L_x^1\to L_x^\infty$ dispersive estimate and Sobolev embedding,
\[
\| F_2(t)\|_{L_x^\infty} \lesssim \int_{I_2} |t-s|^{-\frac32} \|u\|_{L_t^\infty H_x^1}^3 \,ds \lesssim (t-T+\eps^{-\frac14})^{-\frac12}. 
\]
Thus 
\[
\|F_2\|_{L_t^4 L_x^\infty([T,\infty)\times\R^3)} \lesssim  \eps^{\frac1{16}}, \qtq{whence} \|F_2\|_{L_t^4 L_x^6([T,\infty)\times\R^3)}\lesssim \eps^{\frac{1}{32}}. 
\]

Collecting the estimates above, we find
\[
\| e^{i(t-T)\Delta} u(T)\|_{L_t^4 L_x^6([T,\infty)\times\R^3)}\lesssim \eps^{\frac{1}{32}}. 
\]
Choosing $\eps$ sufficiently small, one can then use a continuity argument to deduce
\[
\|u\|_{L_t^4 L_x^6([T,\infty)\times\R^3)}\lesssim \eps^{\frac{1}{32}}.
\]
By standard arguments, such a bound suffices to establish scattering.\end{proof}

\subsection{Variational analysis}\label{S:variational} We briefly review some of the variational analysis related to the ground state $Q$.  For more details, see \cite{HR, Weinstein}.

The ground state $Q$ optimizes the sharp Gagliardo--Nirenberg inequality:
\[
\| f\|_{L_x^4}^4 \leq C_0 \| f\|_{L_x^2} \|f\|_{\dot H_x^1}^3. 
\]
The Pohozaev identities for $Q$ (which arise from multiplying \eqref{elliptic} by $Q$ and $x\cdot\nabla Q$)  imply that
\begin{equation}\label{poho}
 \|Q\|_{L_x^2}\|Q\|_{\dot H_x^1} = \tfrac43C_0^{-1}\qtq{and} M(Q)E(Q)=\tfrac{8}{27}C_0^{-2}.
\end{equation}

\begin{lemma}[Coercivity I] If $M(u_0)E(u_0)<(1-\delta)M(Q)E(Q)$ and $\|u_0\|_{L_x^2}\|u_0\|_{\dot H_x^1} \leq \|Q\|_{L_x^2}\|Q\|_{\dot H_x^1}$, then there exists $\delta'=\delta'(\delta)>0$ so that
\[
\|u(t)\|_{L_x^2}\|u(t)\|_{\dot H_x^1} <(1-\delta')\|Q\|_{L_x^2}\|Q\|_{\dot H_x^1}
\]
for all $t\in I$, where $u:I\times\R^3\to\C$ is the maximal-lifespan solution to \eqref{nls}.  In particular, $I=\R$ and $u$ is uniformly bounded in $H^1$. 
\end{lemma}

\begin{proof} By the sharp Gagliardo--Nirenberg inequality and the conservation of mass and energy,
\[
(1-\delta)M(Q)E(Q) \geq M(u)E(u) \geq \tfrac12\|u(t)\|_{L_x^2}^2 \|u(t)\|_{\dot H_x^1}^2 -\tfrac14 C_0\|u(t)\|_{L_x^2}^3\|u(t)\|_{\dot H_x^1}^3
\]
for $t\in I$. Using \eqref{poho}, this becomes
\[
1-\delta \geq 3\biggl(\frac{\|u(t)\|_{L_x^2} \|u(t)\|_{\dot H_x^1}}{\|Q\|_{L_x^2}\|Q\|_{\dot H_x^1}}\biggr)^2 - 2\biggl(\frac{\|u(t)\|_{L_x^2} \|u(t)\|_{\dot H_x^1}}{\|Q\|_{L_x^2}\|Q\|_{\dot H_x^1}}\biggr)^3.
\]
Using a continuity argument and the fact that
\[
1-\delta \geq 3y^2 -2y^3 \implies |y-1|\geq \delta' \qtq{for some} \delta'>0,
\]
the first statement of the lemma follows. The second statement follows by noting that the $L^2$-norm is conserved and recalling the $H^1$ blowup criterion. \end{proof}

\begin{lemma}[Coercivity II]\label{coercive2} Suppose $\|f\|_{L_x^2}\|f\|_{\dot H_x^1}<(1-\delta)\|Q\|_{L_x^2}\|Q\|_{\dot H_x^1}$.  Then there exists $\delta'=\delta'(\delta)>0$ so that
\[
\|f\|_{\dot H_x^1}^2 - \tfrac34 \|f\|_{L_x^4}^4 \geq \delta' \|f\|_{L_x^4}^4. 
\]
\end{lemma} 

\begin{proof} Write
\[
\|f\|_{\dot H_x^1}^2 - \tfrac34 \|f\|_{L_x^4}^4 = 3E(f) - \tfrac12\|f\|_{\dot H_x^1}^2.
\]
By the sharp Gagliardo--Nirenberg inequality and \eqref{poho}, 
\begin{align*}
E(f)&\geq \tfrac12\|f\|_{\dot H_x^1}^2[1-\tfrac12 C_0\|f\|_{L_x^2}\|f\|_{\dot H_x^1} ] \\
&\geq \tfrac12\|f\|_{\dot H_x^1}^2[1-\tfrac12(1-\delta)C_0\|Q\|_{L_x^2}\|Q\|_{\dot H_x^1}] \geq (\tfrac16+\tfrac{\delta}{3})\|f\|_{\dot H_x^1}^2. 
\end{align*}
Thus
\[
\|f\|_{\dot H_x^1}^2 - \tfrac34 \|f\|_{L_x^4}^4 \geq \delta \|f\|_{\dot H_x^1}^2,\qtq{which implies}
\|f\|_{\dot H_x^1}^2 - \tfrac34\|f\|_{L_x^4}^4 \geq \tfrac{3\delta}{4(1-\delta)}\|f\|_{L_x^4}^4,
\]
as desired. \end{proof}

%%%%%%%%%%%%%%
\section{Proof of Theorem~\ref{T}}  Throughout this section, we suppose $u$ a solution to \eqref{nls} satisfying the hypotheses of Theorem~\ref{T}.  In particular, using the results of Section~\ref{S:variational}, we have that $u$ is global and uniformly bounded in $H^1$, and there exists $\delta>0$ such that
\begin{equation}\label{below}
\sup_{t\in\R}\|u(t)\|_{L_x^2}\|u(t)\|_{\dot H_x^1} < (1-2\delta)\|Q\|_{L_x^2} \|Q\|_{\dot H_x^1}. 
\end{equation}

We will prove that the potential energy of $u$ escapes to spatial infinity as $t\to\infty$.  The same is true for the kinetic energy, but we do not need that here.
\begin{proposition}[Energy evacuation]\label{pee} There exists a sequence of times $t_n\to\infty$ and a sequence of radii $R_n\to\infty$ such that
\[
\lim_{n\to\infty} \int_{|x|\leq R_n} |u(t_n,x)|^4\,dx = 0.
\]
\end{proposition}

Using Proposition~\ref{pee}, we can quickly prove Theorem~\ref{T}.  We only consider the case of scattering forward in time, as the other case is similar.

\begin{proof}[Proof of Theorem~\ref{T}] By Section~\ref{S:variational}, $u$ is global and uniformly bounded in $H^1$.  Fix $\eps$ and $R$ as in Lemma~\ref{criterion}.  Now take $t_n\to\infty$ and $R_n\to\infty$ as in Proposition~\ref{pee}.  Then, choosing $n$ large enough that $R_n\geq R$, H\"older's inequality gives
\[
\int_{|x|\leq R} |u(t_n,x)|^2 \,dx \lesssim R^{\frac32} \biggl(\int_{|x|\leq R_n} |u(t_n,x)|^4\,dx \biggr)^{\frac12} \to 0\qtq{as}n\to\infty. 
\]
In particular, Lemma~\ref{criterion} implies that $u$ scatters forward in time. \end{proof}

We prove Proposition~\ref{pee} by a virial/Morawetz estimate.  We use the virial weight in a large ball around the origin, exploiting \eqref{below} and coercivity to get a suitable lower bound.  Spatial truncation is necessary because our solutions are merely in $H^1$.  The large-radii terms will be treated as error terms.  One needs some compactness to deduce estimates that are uniform in time.  Holmer and Roudenko \cite{HR} employed concentration compactness to reduce to the study of solutions with pre-compact orbit in $H^1$.  We will instead employ the radial Sobolev embedding and use the standard Morawetz weight at large radii.

We first need a lemma that gives \eqref{below} on sufficiently large balls, so that we can exhibit the necessary coercivity. We define $\chi$ to be a smooth cutoff to the set $\{|x|\leq 1\}$ and set $\chi_R(x) := \chi(\frac{x}{R})$ for $R>0$. 

\begin{lemma}[Coercivity on balls]\label{coercive-ball} There exists $R=R(\delta,M(u),Q)>0$ sufficiently large that
\[
\sup_{t\in\R} \|\chi_R u(t)\|_{L_x^2} \|\chi_R u(t)\|_{\dot H_x^1} < (1-\delta)\|Q\|_{L_x^2} \|Q\|_{\dot H_x^1},
\]
In particular, by Lemma~\ref{coercive2}, there exists $\delta'=\delta'(\delta)>0$ so that
\[
\|\chi_R u(t)\|_{\dot H_x^1}^2 - \tfrac34 \|\chi_R u(t)\|_{L_x^4}^4 \geq \delta' \|\chi_R u(t)\|_{L_x^4}^4
\]
uniformly for $t\in\R$. 
\end{lemma}

\begin{proof} First note that 
\[
\| \chi_R u(t)\|_{L_x^2} \leq \|u(t)\|_{L_x^2}
\]
uniformly for $t\in\R$.  Thus, it suffices to consider the $\dot H^1$ term.  For this, we will make use the following identity, which can be checked by direct computation:
\begin{equation}\label{ibp}  
\int \chi_R^2 |\nabla u|^2 \,dx = \int |\nabla (\chi_R u)|^2 + \chi_R\Delta(\chi_R) |u|^2 \,dx. 
\end{equation}
In particular,
\[
\|\chi_R u\|_{\dot H_x^1}^2 \leq \|u\|_{\dot H_x^1}^2 + \mathcal{O}\bigl( \tfrac{1}{R^2} M(u)\bigr). 
\]
Choosing $R$ sufficiently large depending on $\delta,$ $M(u)$ and $Q$, the result follows. \end{proof}

To prove our virial/Morawetz estimate, we will use the following general identity, which follows by computing directly using \eqref{nls}. 

\begin{lemma}[Morawetz identity]\label{morid} Let $a:\R^3\to\R$ be a smooth weight.  Define
\[
M(t) = 2\Im \int \bar u \nabla u \cdot \nabla a \,dx.
\]
Then 
\[
\tfrac{d}{dt} M(t) = \int- |u|^{4} \Delta a + |u|^2(-\Delta\Delta a) + 4\Re a_{jk}\bar u_j u_k\,dx,
\]
where subscripts denote partial derivatives and repeated indices are summed.
\end{lemma}

Let $R\gg1$ to be determined below.  We take $a$ to be a radial function satisfying
\[
a(x) = \begin{cases} |x|^2 & |x|\leq \tfrac{R}{2}, \\ R|x| & |x|>R.\end{cases} 
\]
In the intermediate region $\tfrac{R}{2}<|x|\leq R$, we impose that
\[
\partial_r a\geq 0, \quad \partial_r^2 a\geq 0,\quad |\partial^\alpha a(x)| \lesssim_\alpha R|x|^{-|\alpha|+1}\qtq{for} |\alpha|\geq 1. 
\]
Here $\partial_r$ denotes the radial derivative, i.e. $\partial_r a = \nabla a \cdot \tfrac{x}{|x|}$.  Under these conditions, the matrix $a_{jk}$ is non-negative.

Note that for $|x|\leq \tfrac{R}{2}$, we have
\[
a_{jk}=2\delta_{jk}, \quad \Delta a = 6,\quad \Delta\Delta a = 0, 
\]
while for $|x|>R$, we have
\[
a_{jk}=\tfrac{R}{|x|}\bigl[\delta_{jk}-\tfrac{x_j}{|x|}\tfrac{x_k}{|x|}\bigr],\quad \Delta a = \tfrac{2R}{|x|}, \quad\Delta\Delta a  = 0.
\]

\begin{proposition}[Virial/Morawetz estimate]\label{VM} Let $T>0$. For $R=R(\delta, M(u), Q)$ sufficiently large,
\[
\tfrac{1}{T}\int_0^T \int_{|x|\leq R} |u(t,x)|^4 \,dx \,dt \lesssim_{u,\delta} \tfrac{R}{T} + \tfrac{1}{R^2}. 
\]

\end{proposition}

\begin{proof} Choose $R=R(\delta,M(u),Q)$ as in Lemma~\ref{coercive-ball}.  We define the weight $a$ as above and define $M(t)$ as in Lemma~\ref{morid}.  Note that by Cauchy--Schwarz, the uniform $H^1$-bounds for $u$, and the choice of weight, we have
\[
\sup_{t\in\R} |M(t)|\lesssim_u R.
\]  

We compute
\begin{align}
\tfrac{d}{dt}M(t) & = 8\int_{|x|\leq \frac{R}{2}} |\nabla u|^2 - \tfrac34 |u|^4 \,dx \label{dt1} \\
&\quad + \int_{|x|>R}  - \tfrac{2R}{|x|}|u|^4 + \tfrac{4R}{|x|} |\slashed{\nabla} u|^2 \,dx \label{dt2} \\
&\quad  + \int_{\frac{R}{2}<|x|\leq R} 4\Re a_{jk}\bar u_j u_k +\mathcal{O}\bigl(\tfrac{R}{|x|}|u|^4 + \tfrac{R}{|x|^3}|u|^2\bigr) \,dx, \label{dt3}
\end{align} 
where $\slashed{\nabla}$ denotes the angular part of the derivative.  In fact, as $u$ is radial, this term is zero.  We consider \eqref{dt1} as the main term and \eqref{dt2} and \eqref{dt3} as error terms.  

In \eqref{dt1} we may insert $\chi_R^2$; using \eqref{ibp} as well, we then write
\[
\eqref{dt1} =8\bigl[\|\chi_R u\|_{\dot H_x^1}^2 - \tfrac34\|\chi_R u\|_{L_x^4}^4\bigr] +\int\mathcal{O}\bigl(\tfrac{1}{R^2}|u|^2\bigr)\,dx + \int\mathcal{O}\bigl(\chi_R^4 - \chi_R^2\bigr)|u|^4 \,dx. 
\]

We now apply the fundamental theorem of calculus on an interval $[0,T]$, discard positive terms, use Lemma~\ref{coercive-ball}, and rearrange.  This yields
\begin{align*}
\int_0^T \int \delta' |\chi_R u(t,x)|^4 \,dx\,dt \lesssim \sup_{t\in[0,T]}|M(t)| + \int_0^T \int_{|x|>R} |u(t,x)|^4 \,dx + \tfrac{T}{R^2}M(u).
\end{align*}
By the radial Sobolev embedding, we have
\[
\int_{|x|>R}|u(t,x)|^4 \,dx \lesssim \tfrac{1}{R^2} \|u\|_{L_t^\infty H_x^1}^2 M(u).
\]
Continuing from above, we deduce
\[
\tfrac{1}{T}\int_0^T\int_{|x|\leq R} |u(t,x)|^4\,dx\,dt \lesssim_{u,\delta} \tfrac{R}{T} + \tfrac{1}{R^2},
\]
as desired.\end{proof}

\begin{remark} The idea of patching together the virial weight and the standard Morawetz weight is originally due to Ogawa and Tsutsumi \cite{OgaTsu}.  If one worked only with a truncated virial weight, one would encounter the error term
\[
\int_{|x|>R} |\nabla u(t,x)|^2\,dx.
\]
If $u$ has a pre-compact orbit in $H^1$, this term can be made small uniformly for $t\in\R$.  Otherwise, it is not clear that this can be achieved.
\end{remark}

\begin{proof}[Proof of Proposition~\ref{pee}]  Applying Proposition~\ref{VM} with $T$ sufficiently large and $R=T^{1/3}$ implies
\[
\tfrac{1}{T}\int_0^{T}\int_{|x|\leq t^{1/3}} |u(t,x)|^4\,dx\,dt \lesssim T^{-2/3},
\]
which suffices to give the desired result. \end{proof}

\end{document}